\documentclass[11pt]{article}
\usepackage[top=1in,bottom=1in,left=1.5in,right=1.5in]{geometry}
\usepackage{amsmath,amssymb,amsthm}
\usepackage{authblk}
\usepackage{latexsym}% defines $\Box$ for LaTeX2e
\usepackage{graphicx}

\newtheorem{theorem}{Theorem}[section]
\theoremstyle{plain}
\newtheorem{conjecture}[theorem]{Conjecture}
\newtheorem{lemma}[theorem]{Lemma}

\newtheorem{corollary}[theorem]{Corollary}

\newtheorem{question}[theorem]{Question}
\newtheorem{problem}[theorem]{Problem}
 
\theoremstyle{definition}

\newtheorem{example}[theorem]{Example}

\newcommand\R{\mathord{\mathbb R}}
\newcommand\C{\mathord{\mathbb C}}
\newcommand\Z{\mathord{\mathbb Z}}
\renewcommand{\c}{\mathbf{c}\mathnormal}

\newcommand{\f}{\mathbf{f}}
\newcommand{\w}{\mathbf{w}}
\newcommand{\x}{\mathbf{x}}
\newcommand{\y}{\mathbf{y}}
\newcommand{\z}{\mathbf{z}}
\newcommand{\0}{\mathbf{0}}
\newcommand{\1}{\mathbf{1}}
\newcommand\bd{{\bf d}\mathnormal}

\newcommand\cP{{\cal P}}
\def\diag{\mathop{{\rm diag}}\nolimits}

\newcommand{\trans}{^\top}
\newcommand{\opt}{\mathop{\mathrm{opt}}\nolimits}

\newcommand{\gam}{\boldsymbol{\gamma}}
\newcommand{\omeg}{\mbox{\boldmath{$\omega$}}}
\newcommand{\adj}{\mathrm{adj\;}}
\newcommand{\rank}{\mathrm{rank\;}}
\newcommand{\set}[1]{\{#1\}}

\begin{document}

\title{On 1-sum flows in undirected graphs}
\author[1,4]{S.~Akbari}
 \author[2]{S.~Friedland\thanks{The work of S. Friedland was supported by the NSF grant DMS-1216393.}}
\author[3]{K.~Markstr\"{o}m}
\author[4]{S.~Zare}
\affil[1]{Department of Mathematical Sciences, Sharif University of Technology, 11155-9415,
Tehran, Iran. E-mail:\texttt{s\_akbari@sharif.edu}}
\affil[2]{Department of Mathematics, Statistics and Computer Science,
 University of Illinois at Chicago, Chicago Illinois 60607-7045, USA. E-mail:\texttt{friedlan@uic.edu}}
\affil[3]{Department of Mathematics and Mathematical Statistics, Ume\aa\ University, SE-901 87, Ume\aa, Sweden.
 E-mail:\texttt{Klas.Markstrom@math.umu.se}}
\affil[4]{School of Mathematics, Institute for Research in Fundamental Sciences (IPM), 19395-5746, Tehran, Iran. E-mail: \texttt{sa\_zare\_f@yahoo.com} }

\maketitle

\date{}

\begin{abstract} 
Let $G=(V, E)$ be a simple  undirected graph. For a given set $L\subset \mathbb{R}$, a function $\omega: E \longrightarrow L$  is called an $L$-flow.
Given a vector $\gamma \in \mathbb{R}^V,$ we say that $\omega$ is a $\gamma$-$L$-flow if for each $v\in V$, the sum of the values on the edges 
incident to $v$ is $\gamma(v)$. If $\gamma(v)=c$, for all $v\in V$, then the $\gamma$-$L$-flow is called a $c$-sum $L$-flow.  In this paper we study 
the existence of $\gamma$-$L$-flows for various choices of sets $L$ of real numbers, with an emphasis on 1-sum flows.

Given  a natural $k$ number, a {\it $c$-sum $k$-flow} is a $c$-sum flow with values from the set \mbox{$\{\pm 1,\ldots ,\pm(k-1)\}$}. Let $L$ be a subset of real numbers containing $0$ and denote $L^*:=L\setminus \{0\}$. Answering a question from \cite{akb}  we characterize which bipartite graphs admit a $1$-sum $\mathbb{R}^*$-flow or a $1$-sum $\mathbb{Z}^*$-flow.  We also show that that every $k$-regular graph, with $k$ either odd or congruent to 2 modulo 4, admits a $1$-sum $\{-1, 0, 1\}$-flow.

\end{abstract} 

 \noindent \emph{Keywords}:  $L$-flow,  $\gamma$-$L$-flow, c-sum flow, bipartite graph.

 \noindent {\bf 2010 Mathematics Subject Classification}: 0521, 90C05.

%---------------------------------------------------------------------------------------------------------------------------------------------------------------------------------------------------------------------------
 \section{Introduction}
Let $G=(V,E)$ be a simple undirected graph with $n=|V|$ vertices and $m=|E|$ edges. We say that a vertex $v\in V$ and an edge $e\in E$ are 
incident if $e=\{v,u\}$. Assign a weight $\omeg:E\to \R$.  In this paper we view $\omeg$ as a flow in $G$.
The value of $\omeg$ at $v\in V$, denoted as $\gamma(v)$, is given by $\gamma(v)=\sum_{e\in E,v\in e} \omega(e)$.
By abuse of notation we view $\omeg=(\omega(e))_{e\in E}, \gam=(\gamma(v))_{v\in V}$ as column vectors in $\R^E,\R^V$, 
respectively.
For a given set $L\subset \R$, $\omeg$ is called an $L$-flow if $\omeg:E\to L$.  Thus an $\R$-flow is just a flow defined above.
Given a vector $\gam\in \R^{V}$ we say that $\omeg$ is a $\gam$-$L$-flow if the value of an $L$-flow $\omeg$ on each vertex $v$
is $\gam(v)$.   Let $c\in\R$.  Then $\gam$-$L$-flow is called a $c$-sum $L$-flow if $\gam(v)=c$ for all $v\in V$ and $\omeg:E\to L$.

In this paper we study the existence problem of $\gam$-$L$-flow on undirected graphs.   
The problem of finding $c$-sum $S$-flows was studied in the papers \cite{akb1,akb2, akb3,akb}.
For simplicity of exposition we will assume that $G$ is a connected graph. 

The existence of $\gamma$-$\R$-flow is a linear algebra problem.   Let $L\subset \R$ be an interval.  (It may be open, closed, half open, finite or infinite.)
Then the existence of $\gamma$-L-flow is a problem linear programming related to graphs.  See for example \cite{CCPS}.
%

%---------------------------------------------------------------------------------------------------------------------------------------------------------------------------------------------------------------------------
\section{Existence of $\gam$-interval-flows}
Given a value $\gamma$ and an interval $L$ the most basic is whether a graph $G$ has a $\gam$-$L$-flow or not.  If $L$ is the entire real line this 
a purely linear algebraic question, and when $L$ is a proper subinterval of reals we can apply methods from linear programming to find conditions for its solvability.  In this section we will first strengthen an existence result from \cite{akb} for $\gam$-$\R$-flows and then look at the case when $L$ is a proper subinterval.

%------------------------------------------------------------------------------------------------------------------
\subsection{Existence of $\gam$-$\R$-flows}

Let $G=(V,E)$ be a simple undirected graph. Let $A(G):=A=[a_{ve}]\in\R^{V\times E}$ be the vertex edge incidence matrix of $G$. 
That is $a_{ve}=1$ if $v\in e$ and $a_{ve}=0$ otherwise.   It is well known that $A$ is unimodular, i.e. all its minors have values in the set $\{-1,0,1\}$,
if and only if $G$ is bipartite \cite{Ege}.  (See \cite[\S6.5]{CCPS} for a textbook reference.)  Assume that $G$ is connected.  
Then $\rank A=n$ if $G$ contains an odd cycle and $\rank A=n-1$ if $G$ is bipartite
\cite[p. 63]{CRS}.  The following result is a more detailed version of the result proved in \cite{akb}.  
\begin{lemma}\label{existgamRflow}  
	Let $G=(V,E)$ be a connected graph and  $\gamma\in \R^v$ is given.  Then
	\begin{enumerate}
		\item If $G$ is not bipartite then there exists a $\gam$-$\R$-flow. Furthermore, if $\gam\in \Z^V$ then there exists 
		a solution $\omeg$ such that $2\omeg\in\Z^E$.

		\item Assume that $G$ is bipartite and $V=V_1\cup V_2$ is the bipartite decomposition of vertices of $G$.
			Then there exists a $\gam$-$\R$-flow if and only if
			\begin{equation}\label{bipgamcond}
				\sum_{v\in V_1}\gamma(v) -\sum_{v\in V_2}\gamma(v)=0.
			\end{equation}
			Equivalently, let $\y=(y_v)_{v\in V}\in \R^V$ be a vector such that $y_v=1$ if $v\in V_1$
			and $y_v=-1$ if $v\in V_2$.  That is, $\y\trans=(\1_{V_1}\trans, -\1_{V_2}\trans)$.  Then $\y$ is a basis of the null
			space of $A(G)\trans$.  Furthermore, if $\gam\in\Z^V$ and the condition \eqref{bipgamcond} holds then there exists a solution $\omeg\in\Z^E$.
	\end{enumerate}
\end{lemma}
\begin{proof}  
	Recall that the existence of $\gam$-$\R$-flow is equivalent to the solvability of the system:
	\begin{equation}\label{sysomeggam}
		A(G)\omeg = \gam.
	\end{equation}
	\begin{enumerate}
	
	\item  $G$ is not bipartite if and only if it contains an odd cycle.  So $\rank A(G)=|V|=n$.  Hence \eqref{sysomeggam} is solvable.
		We now show that if $\gam\in\Z^V$ then there exists a solution $\omeg$ such that $2\omeg\in \Z^V$.
		Since $G$ is not bipartite it contains an odd cycle $C$.
		
		First, assume  that $C$ is a Hamiltonian cycle.  We can assume that $E(C)=\{\{v_1,v_2\},\ldots \{v_{n-1},v_n\}, \{v_n.v_1\}\}$. We claim that 
		$\det A(C)=2$.  Indeed, $A(C)=P+P\trans$, where $P$ is the permutation matrix corresponding to an odd cycle $v_1\to v_2\to\ldots\to  v_n\to v_1$.  
		So 
		\[\det A(C)=\det (P+P\trans)=\det (P\trans (P^2+I))=\det P\trans \det (I+P^2).\]
		Since $P$ is a cyclic matrix of an odd order $\det P=\det P\trans =1$.  Note that $P^2$ corresponds also to a cyclic matrix of order $n$.
		That is, $P^2$ is similar to $P$.  Hence $\det (I+P^2)=\det(I+P)$.  Recall that the eigenvalues of $P$ are all the $n$-th roots of $1$.
		That is, $\det (\lambda I_n -P)=\lambda^n-1$.  Let $\zeta_1,\ldots,\zeta_n$ be $n-th$ roots of $1$.  So
		\[\det (I+P)=\prod_{i=1}^n (1+\zeta_i)=1+\sum_{i=1}^n \sigma_i.\]
		Here, $\sigma_i$ is the $i-th$ elementary polynomial of $\zeta_1,\ldots,\zeta_n$ for $i=1,\ldots,n$.
		As $\zeta_1,\ldots,\zeta_n$ are the roots of $\lambda^n-1=0$, it follows that $\sigma_i=0$ for $i=1,\ldots,n-1$. 
		As $n$ is odd, $\sigma_n=1$.  Hence $\det A(C)=2$. 
		
		Next , assume  that $\omega(e)=0$ if $e\not\in E(C)$.  Let $\omeg'\in \R^{E(C)}$ be the unique solution of $A(C)\omeg'=\gam$.
		So the coordinates of $\omeg$ coincide with coordinates of $\omeg'$ on $E(C)$.
		Clearly, $2\omeg'=2A(C)^{-1}\gam$.  Recall that $A(C)^{-1}=\frac{\adj (A(C))}{\det A(C)} $.
		Here, $\adj(A(C))$ is the adjoint matrix of $A(C)$ whose entries are minors of $A$ of order $n-1$.  Since the entries of $A(C)$ are integers
		it follows that the entries of  $\adj(A(C))$ are integers.   Hence $\frac{2}{\det A(C))}\adj A(C)=\adj(C)$ and $2\omeg'=\adj(A(C))\gam$.   
		So $2\omeg'\in\Z^{E(C)}$ and $2\omeg\in \Z^E$.
		
		We now assume that $C$ is not Hamiltonian.  Let $V(C)=\{v_1,\ldots,v_l\}$ where $l$ is odd and $3\le l <n$.
		Delete the edge $\{v_{l},v_1\}$ from $C$ to obtain a path $Q$.  Extend $Q$ to a spanning tree $T'$ of $G$.
		Let $G'=(V,E(T')\cup \{v_{l},v_1\})$.  So $G'$ has exactly one odd cycle $C$.  We claim that $\det A(G')=\pm 2$.
		Observe if we delete the edges of $E(C)$ in $G'$ we obtain a forest.  Hence, $G'$ contains at least one vertex $u$ of degree $1$.
		Expand $\det A(G')$ by the row corresponding to $u$.  Then $\det A(G')=\pm \det A(G_1)$, where $G_1$ is obtained from $G'$
		by deleting the vertex $u$.  Continue this process to deduce that $\det A(G')=\pm \det A(C)=\pm 2$.
		
		Let $\omeg$ be the unique solution of \eqref{sysomeggam} where $\omega(e)=0$ if $e\not\in E(G')$.  The above arguments show that
		$2\omeg\in \Z^E$.

	\item Assume that $G$ is bipartite and $V=V_1\cup V_2$ is the bipartite decomposition of $V$.
		Clearly, $\y\trans A(G)=0$.  Since $\rank A(G)=n-1$ then $\y$ spans the null space of $A(G)\trans$.  Hence the system 
		\eqref{sysomeggam} is solvable if and only if the condition \eqref{bipgamcond} holds.
		
		Let $\gam\in\Z^V$ and assume that the condition \eqref{bipgamcond} holds.  We now construct a solution $\omeg\in\Z^E$.
		Let $T'$ be a spanning tree of $G$.  Let $\omeg$ be the unique solution of \eqref{sysomeggam} such that $\omega(e)=0$ 
		if $e\not\in E(T')$.  Recall that $\rank A(T')=n-1$.  Since $\y$ spans the null space of $A(T')\trans$ it follows any $n-1$ rows of $A(T')$ are linearly 
		independent.  Let $B$ be a square submatrix of $A$ obtained by deleting a row in $A(T')$ corresponding to a vertex $v \in V$.  Denote by $\gamma'$ 
		the vector obtained from $\gamma$ be deleting coordinate $\gamma(v)$.    As $A(T')$ unimodular it follows that $\det B=\pm 1$.
		Hence the solution of the system $A(T')\omeg'=\gam$ is given by $\omeg=B^{-1}\gam\in\Z^{E(T')}$.  As $\omega(e)=\omega(e')$ for each
		$e\in E(T')$ we deduce that $\omega\in \Z^E$.
	
	\end{enumerate}

\end{proof}

%------------------------------------------------------------------------------------------------------------------
\subsection{Linear programming conditions for the existence of $\gam$-interval-flows}\label{sec:ex}
In this section we apply linear programming methods to study the conditions for existence of $\gam$-$L$-flow, where $L$
is an interval of $\R$.  For simplicity of exposition we assume that $L$ is a closed bounded interval $[a,b]$. 
Our methods and arguments are closed to those given in \cite{CCPS}.

We denote $[n]=\{1,\ldots,n\}$.  We will identify
\[V\equiv [n] \quad E\equiv [m], \quad \R^V\equiv\R^n,\quad \R^E\equiv \R^m,\]
and no ambiguity will arise.  Let $\1_m=\1_E$ be a column vector with $m=|E|$ coordinates equal to 1.  
For two vectors $\x=(x_1,\ldots,x_n)\trans,\z=(z_1,\ldots,z_n)\trans\in \R^m$
we denote $\x\le \y$ if $x_j\le y_j$ for $j=1,\ldots,m$.

We are looking for a solution of  \eqref{sysomeggam} such that
\begin{equation}\label{1flowpmsol}
	a\1_m\le \omeg\le b\1_m.
\end{equation}
Denote by $I_m$ the identity matrix of order $m$ and by $\R_+$ the set of nonnegative real numbers.
Let $\bd(G)=(\deg(v))_{v\in V}\in\R^V$ be the degree sequence of $G$.  Note that $\bd(G)=A(G)\1_m$.
\begin{lemma}\label{Farkaslem}  
	Consider the system \eqref{sysomeggam} satisfying the conditions \eqref{1flowpmsol}. Then this system of equations is solvable 
	if and only if the following conditions are satisfied: 
	For each $\w\in\R_+^m$ one has the inequality 
	\begin{equation}\label{Farkaslemma1}
		\max\{(\gam-a\bd(G))\trans \z,\; \z \in \R^n,  A\trans\z\le \w\}\le (b-a) \1_m\trans \w.
	\end{equation} 
\end{lemma}
\begin{proof}   
	Clearly, the system \eqref{sysomeggam} satisfying the conditions \eqref{1flowpmsol} is equivalent to the following conditions.
	\begin{equation}\label{ineqform1flowpm1}
		F\x\le \f, \; \x\in \R^{m}, \textrm{ where } F=\left[\begin{array}{r}A\\-A\\I_m\\-I_m\end{array}\right],\;
		\f= \left[\begin{array}{r}\gam\\-\gam\\b\1_m\\-a\1_m\end{array}\right].
	\end{equation}
	Farkas lemma claims \cite{CCPS} that the above system is solvable if and only if the following implication holds: 
	\begin{equation}\label{Farkascond1}
		\y\in\R_+^{2(n+m)} \textrm{ and } \y\trans F=\0\trans  \Rightarrow \y\trans \f\ge 0, 
	\end{equation}
	where $\y\trans=(\y_1\trans,\y_2\trans, \y_3\trans,\y_4\trans),\y_1,\y_2\in\R^n, \y_3,\y_4\in\R^m$.
	The equation $\y\trans F=\0\trans$ is equivalent to 
	\begin{equation}\label{Farkascond2}
		\y_4=\y_3-A\trans \z, \quad \z=\y_2-\y_1.
	\end{equation}
	The condition $\y\ge \0$ is equivalent to the inequalities 
	\begin{equation}\label{Farkascond3}
		\y_3\ge \0, \quad \y_3\ge A\trans \z.
	\end{equation}
	(Note that if the above conditions hold, one can always choose $\y_1,\y_2\ge \0$ such that $\z=\y_2-\y_1$.)
	Clearly, these conditions are satisfiable for $\y_3\ge 0$ and $\z=0$.
	Finally, the condition $ \y\trans \f\ge 0$ is equivalent to the following the inequality
	\[  \z\trans \gam-a\z\trans A\1_m\le (b-a)\y_3\trans \1_m.\]
	Set $\w=\y_3$ and recall that $A\1_m=\bd(G)$ to deduce the lemma.
\end{proof}

The condition \eqref{Farkaslemma1} can be stated as the following nonlinear inequality in $\z\in\R^n$. Let $\z=(z_1,\ldots,z_n)\trans$ be an 
arbitrary vector in $\R^n$.  Define $\w(\z)=(w_1(\z),\ldots,w_m(\z))\trans\in \R+^m$ as follows:

\begin{equation}\label{defw}
	w_j(\z)=\max(0, (A\trans \z)_j) \textrm{ for } j=1,\ldots,m.
\end{equation} 
Then the condition \eqref{Farkaslemma1} is equivalent to
\begin{equation}\label{Farkaslemma2}
	(\gam-a\bd(G))\trans \z\le (b-a) \1_m\trans\w(\z) \textrm{ for each } \z\in\R^n.
\end{equation}

We state an equivalent necessary and sufficient condition for solvability of the system \eqref{sysomeggam} satisfying the conditions \eqref{1flowpmsol}
which can be stated in terms of nonnegative solutions of a corresponding variant of \eqref{sysomeggam}.
\begin{lemma}\label{equivcondsolv}  
	The following are equivalent:
	\begin{enumerate}
		\item The system \eqref{sysomeggam} satisfying the conditions \eqref{1flowpmsol} is solvable.
		\item The system 
			\begin{equation}\label{equivcondsolv1}
				A\omeg'=\gam- a\bd(G), \quad \0\le \omeg'\le (b-a)\1_m
			\end{equation}
			is solvable.
	\end{enumerate} 
\end{lemma}
\begin{proof} 
	The proof is straightforward by noting that $\omega$ is a solution satisfying \eqref{sysomeggam}-\eqref{1flowpmsol} if and only if 
	$\omeg'=\omeg-a\1_m$ satisfies \eqref{equivcondsolv1}.
\end{proof}

We now give the condition for the existence of nonnegative solutions of \eqref{sysomeggam}.
\begin{lemma}\label{nonnegsol}
	Consider the system \eqref{sysomeggam} with $\gam\ne \0$.  Then this system has a nonnegative solution if and only if
	\begin{equation}\label{nonnegso1}
		\min\{ \gam\trans \z, \;\z\in\R^n, A\trans\z\ge \0\}=0
	\end{equation}
\end{lemma}
\begin{proof} 
	As in the proof of Lemma \ref{Farkaslem} the existence of nonnegative solutions of the system \eqref{sysomeggam} is equivalent to the 
	system
	\[F\x\le \f, \; \x\in \R^{m}, \textrm{ where } F=\left[\begin{array}{r}A\\-A\\-I_m\end{array}\right],\;
	\f= \left[\begin{array}{r}\gam\\-\gam\\\0\end{array}\right].\]
	The above system is solvable if and only if each nonnegative solution of $\y\trans F=\0\trans  $ satisfies the inequality $\y\trans \f\ge 0$.
	Let $\y\trans=(\y_1\trans,\y_2\trans,\y_3)$, where $\y_1,\y_2\in\R^n$ and $\y_3\in \R^m$.
	Then the condition $\y\ge \0$ and $F\trans \y=0$ are equivalent to the condition that $\y_3=A\trans \z\ge 0$, where $\z=\y_1-\y_2$.
	The condition $\y\trans \f\ge 0$ is equivalent to $\gam\trans \z\ge 0$.  Note that if we choose $\z=\0$ then $\y_3=\0$ and $\gam\z=0$.
	This implies \eqref{nonnegso1}.
\end{proof}

We now restate our results for $c$-$[a,b]$-flows.  That is, we let $\gam=c\1_n$.
\begin{theorem}\label{char1pm1flow}  Let $G=(V,E)$ be a simple undirected graph with no isolated vertices.  The following are equivalent:
	\begin{enumerate}
		\item $G$ has $c$-$[a,b]$-flow.
		\item If  $G$ has a nonnegative $\c\1_m - a\bd(G)$-flow such that the value of this flow on each edge is at most $b-a$.
		\item For each $\w\in\R_+^m$ one has the inequality 
			\[\max\{(c\1_n-a\bd(G))\trans \z,\; \z \in \R^n,  A\trans\z\le \w\}\le (b-a) \1_m\trans \w.\]
	\end{enumerate}
\end{theorem}
%

%---------------------------------------------------------------------------------------------------------------------------------------------------------------------------------------------------------------------------
\section{The range of a 1-flow}
Once a graph has been shown to have a $\gam$-$\R$-flow it is natural to ask which values the edge weights in such a flow can take. In this 
section we will look at questions of this type for the specific case of 1-sum flows.   Given a 1-sum flow on a graph $G$ we call the smallest interval which 
contains all the edge weights of the flow is called  the range of the flow.  A natural question now is: Given a graph $G$, which is the shortest interval $L$ 
such that $L$ is the range of  a $1$-sum flow on $G$?  Starting from the other end we can also ask for a characterization of the graphs which have a 
1-sum flow with range in some given interval $L$.

We will prove some results of both these forms. First we will look at 1-sum flows on trees, which have a unique 1-sum flow or none at all, and 
find the optimal range for this class of graphs. After that we do the same for graphs with a single cycle, and then give some bounds for the range of 
1-sum flows on general graphs.  After this we instead look at conditions guaranteeing that a graph has a 1-sum $[1,1]$-flow, or a non-negative flow .

%------------------------------------------------------------------------------------------------------------------
\subsection{The range of 1-sum flows on trees}
For a given graph $G=(V,E)$ and the weight function $\omega:E\to\R$, for each subset $Q$ of $E$ we denote by $\omega(Q):=\sum_{e\in Q}\omega(e)$.
We agree that $\omega(\emptyset)=0$.
In this section we analyze the the range of values of $1$-flow on a tree $T=(V,E)$ with $n$ vertices, i.e. $n=|V|$
and we let $V=[n]=\{1,\ldots,n\}$.  
Recall that $m=|E|=n-1$ and $T$ is bipartite.   Let $A=A(T)$.  Then the system \eqref{sysomeggam} is solvable if and only if the condition
\eqref{bipgamcond} holds.  Assume that \eqref{bipgamcond} holds.   
  
We now estimate the coordinates of the solution of  \eqref{sysomeggam}. We perform the following \emph{pruning} procedure of a tree $T$.  Let $T_1=T$ 
and  $P_1\subset V$ be leaves. If $T_1=K_2=K_{1,1}$ or the star $K_{1,n-1}$ then we are done.  Otherwise, let $T_2$ be the subtree of $T_1$ obtained 
by deleting the leaves $P_1$ and the corresponding $|P_1|$ edges.  Denote by $E(P_1)\subset E(T)$ the subset of edges attached to $P_1$. We now continue 
this process on $T_2$. We obtain a sequence of subtrees $T_1\supset T_2\supset\cdots\supset T_k$, where $T_k=K_{1,n_k-1}$.  The leaves of 
$T_i=(V_i,E_i), n_i=|V_i|$ are $P_i$. Then $E(P_i):=E_i\setminus E_{i+1}$ for $i=1,\ldots,k$.  ($E_{k+1}=\emptyset$.)  
Note
\begin{equation}\label{seqpendvert}
	p_1=|P_1|\ge p_2=|P_2|\ge \ldots \ge p_k=|P_k|=\max(2,n_k-1).
\end{equation}
Indeed, if we delete all leaves of $T_1$ which are neighbors of $u$, then it is possible that $u$ is not a leaf of $T_2$.
On the other hand if $u$ is a leaf in $T_2$ then $u$ is not a leaf in $T_1$ and $u$ has at least leaf neighbor in $T_1$.

We consider the system \eqref{sysomeggam}.  Let $\gam^{(1)}=(\gamma^{(1)}(v))_{v\in V_1}=\gam$ and
$\omega_1(e)=\gamma^{(1)}(v)$ for  $e\in E(P_1)$ and $v\in e$.  The values of $\omega_{1}(e)$ is the value of $\omega(e)$,  where $e$ is the unique 
edge in $T_1$ that contains the vertex $v\in P_1$.  

Let $\gam^{(i)}=(\gamma^{(i)}(v))_{v\in V_i}$ and $\omega_{i}(e), e\in E(P_i)$ be defined recursively as follows for $ i=2,\ldots,k$:
\begin{eqnarray}
&&\gamma^{(i)}(v)=\gamma^{(i-1)}(v) \textrm{ for } v\in V_i \textrm{ not connected to } P_{i-1}\notag\\
&&\gamma^{(i)}(v)=\gamma^{(i-1)}(v) - \sum_{e\in E(P_{i-1}), v\in e} \omega_{i-1}(e) \textrm{ for } v\in V_i \textrm{ connected to } P_{i-1}\notag\\
&&\omega_{i}(e)=\gamma^{(i)}(v) \textrm{ for } e\in E(P_i) \textrm{ and }  v\in P_i.
\label{defxlij}
\end{eqnarray}
It is easy to see that each $\gamma_l$ can appear at most in one of the coordinates of $\gamma^{(j)}$ with coefficient $\pm 1$.
(This is also follows from the condition \eqref{bipgamcond}.)
Now consider $T_k$. Assume first that $T_k=K_2$.  So $P_k=\{u,v\}$. In order to be able to solve the original system 
one needs that condition $\gamma^{(k)}(u)-\gamma^{(k)}(v) =0$.  Assume  $T_k=K_{1,n_k-1}$, where $n_k\ge 3$.  Let $u$ be the center of the star.
Then the solvability condition is:
\begin{equation}\label{solcondstar}
	\gamma^{(k)}(u)=\sum_{v\in P_k} \gamma^{(k)}(v).
\end{equation} 
In both cases, since each $\gamma(w)$ appears exactly once in some degree of $T_k$ with coefficient $\pm 1$, we deduce that
this is equivalent to the fact that a basis to the null space of $A(T)\trans$ is $\y\trans=(\1_{V_1}\trans, -\1_{V_2}\trans)$.
 \begin{theorem}\label{1admisflowtree}
	Assume that a tree $T$ has $1$-flow, i.e. $T$ is a balanced bipartite graph. Let $T=T_1\supset\ldots \supset T_k$ be the subtrees defined as above.
	Then the following conditions hold:
	\begin{enumerate}
		\item The unique flow is integer valued.
		
		\item $\omega_1(e)=1$ for $e\in E(P_1)$.  Hence $\omega(E(P_1))=p_1$.

		\item If $i$ is even then $\omega_i(e)\le 0$ for $e\in E(P_i)$.  If $i$ is odd then $\omega_i(e)\ge 1$ for $e\in E(P_i)$.

		\item 
			\begin{equation}\label{sumxji}
				(-1)^i\omega(E(P_i))\ge \sum_{j=0}^{i-1} (-1)^{j} p_{i-j} \textrm{ for }i=2,\ldots,k
			\end{equation}

		\item Let $V(T)=V_1(T)\cup V_2(T)$ be the bipartite decomposition of the balanced tree $T$. Then both $V_1(T)$ and $V_2(T)$ contain a leaf.

		\item If $p_1=2$ then $T$ is a path and $\omega$ is $\{0,1\}$-flow.

		\item If $p_1=3$ the $T$ has the shape ``T'' and $\omega$ is $\{0,1\}$-flow.

		\item Assume that $p_1\ge 4$.  Then $n\ge 6$ and the flow is 1-sum $\{1-\lfloor\frac{p_1}{2}\rfloor, 2-\lfloor\frac{p_1}{2}\rfloor,\ldots,\lfloor\frac{p_1}{2}\rfloor\}$. 

		\item In particular, the flow is in $[2-\frac{n}{2}, \frac{n}{2}-2]$.  The lower bound achieved only for the unique tree $T_{\min}$, where
			$T_{\min}$ is $K_2$ with appended $\frac{n-2}{2}$ vertices to each vertex of $K_2$.  For $T_{\min}$ we obtain that the flow is 1-sum $\{\frac{4-n}{2},1\}$
			flow.  The upper bound is obtained on for the unique tree $T_{\max}$, $T_{\max}$ the path on $4$ vertices, $PL_4$, with $\frac{n-4}{2}$ vertices appended 
			to each leaf of $PL_4$.  For $T_{\max}$ we obtain the flow is 1-sum $\{\frac{6-n}{2},1, \frac{n-4}{2}\}$. 

		\item The other optimal tree $T_{\opt}$, different from $T_{\max}$ and $T_{\min}$, on $n\ge 8$ vertices is obtained as follows.  Take the path
			 $PL_4:=v_1 - v_2 - v_3-v_4 $,  Add $\frac{n-4}{2}$ leaves at $v_1$, $\frac{n-6}{2}$ leaf at $v_4$ and one leaf at
			 $v_2$.  Then this flows is 1-sum $\{\frac{6-n}{2},\frac{8-n}{2}, 1, \frac{n-8}{2},\frac{n-6}{2}\}$.
	\end{enumerate}
\end{theorem}
\begin{proof}  
	\begin{enumerate}

	\item This follows from part \emph{2.} of Lemma \ref{existgamRflow}.
	
	\item Self evident.
	
	\item  Use \eqref{defxlij}, the fact that $\gam=\1_n$ and each $\omega_1(e)=1$ for $e\in E(P_1)$ to deduce that that each 
		$\omega_2(e)\le 0$ for $e\in E(P_2)$. 	Continuing in this manner, using \eqref{defxlij} we deduce the claim.
	
	\item  Let $E'(P_i)$ be the subset of all  edges in $E(P_i)$  which are connected to $P_{i+1}$.
		(Note that some leaves in $T_i$ may be connected to nonleaf vertices in $T_{i+1}$.)
		Then summing the $1$-flow on all vertices in $P_i$, for $i\ge 2$ we get
		\begin{equation}\label{flowPi}
		p_i=\omega_{i-1}(E'(P_{i-1}))+\omega_i(E(P_i)).
		\end{equation}
		Let $i=2$.  As $\omega_1(e)=1$ for each $e\in E(P_1)$ we deduce that
		$p_2\le \omega_1(E(P_1))+\omega_2(E(P_2))=p_1+\omega_2(E(P_2))$.
		This establishes \eqref{sumxji} for $i=2$.  
	
		Assume now that $i=3$.  As $\omega_2(e)\le 0$ for each $e\in E(P_2)$ the equality \eqref{flowPi} and the inequality yields the 
		inequality  \eqref{sumxji} for $i=2$ 	yields:
		\[p_3\ge \omega_2(E(P_2))+\omega_3(E(P_3))\ge p_2-p_1+\omega_3(E(P_2)).\]
		This establishes \eqref{sumxji} for $i=3$.   Continuing in this manner we deduce \eqref{sumxji} for $i=4,\ldots,k$. 
	
	\item  Assume to the contrary that $V_1(T)$ does not have a leaf.  So $n-1=|E(T)|\ge 2 |V_1(T)|=n$ as $T$ is a balanced bipartite.
		This is impossible.  Hence $V_1(T)$ contains a leaf.  Similarly, $V_2(T)$ contains a leaf.
	
	\item  Straightforward.

	\item  Straightforward.

	\item  $\omega\in\Z^{n-1}$ it is enough to show that $\omega$ is 1-sum $[1-\lfloor\frac{p_1}{2}\rfloor, \lfloor\frac{p_1}{2}\rfloor]$ flow.
		We will prove the claim on induction on $n$.  In view of \emph{6.-7.} the claim holds for $n=2,4$.   Assume that the claim holds for all even $n$,
		where $n\le 2N$.  Assume that $n=2N+2$.   In view of \emph{6.-7.} we assume that $p\ge 4$.   Let $T$ be a balanced tree on $2N+2$ vertices
		with $p\ge 4$ leaves.  Let $\omega:E(T)\to \Z$ be the unique $1$-flow on $T$.
		Let  $u\in V_1(T),v\in V_2(T)$ be two leaves of $T$.  Assume that $\{u,u_1\}, \{v,v_1\}\in E(T)$. 
		Take the path $Q$ in $T$ connecting $u$ and $v$ given by $e_1=\{u,u_1\}-e_2-\ldots-e_{2l+1}=\{v_1,v\}$.
		Note that there is the following flow on $Q$: 
		\[\theta(e_1)=\theta(e_3)=\ldots=\theta(e_{2l+1})=1, \quad \theta(e_2)=\ldots=\theta(e_{2l})=-1.\]
		Let $T'$ be the tree obtained from $T$ by deleting the vertices $u,v$.  
		Denote $F=\{e_2,\ldots,e_{2l}\}\subset E(T')$ and assume that $T'$ has $p'$ leaves.
		Let $\omega':T'\to \Z$ be the unique $1$-flow on $T'$.
		Then $\omega'(e)=\omega(e)$ if $e\in E(T')\set F$ and $\omega'(e_j)=\omega(e_j)-\theta(e_j)$ for $j=2,\ldots,2l$.
		So $\omega'(e)-1\le \omega(e)\le \omega'(e)+1$ for each $e\in E(T')$.
	
		Suppose first that $\deg(u_1),\deg(v_1)\ge 3$.  Then $p'=p-2$.  By induction hypothesis
		\[2-\lfloor \frac{p}{2}\rfloor=\1-\lfloor \frac{p-2}{2}\rfloor\le \omega'(e)\le \lfloor \frac{p-2}{2}\rfloor=-1+\lfloor \frac{p}{2}\rfloor.\]
		This proves \emph{8.} in this case.
	
		Suppose now that $\deg(u_1)=2$.  Then $\omega(e_1)=1$ and $\omega(e_2)=0$.  Delete vertices $u,u_1$ in $T$ to obtain
		a balanced tree $T''$ with $2N$ vertices and $p''$ pendant vertices.  Clearly $p''\le p$. Also the $1$-flow on $T''$ coincides.
		Use the induction hypothesis to deduce \emph{8}.

	\item[9-10]  Clearly the maximal number of leaves in a balanced tree is $p_1=n-2$.  This equality is achieved only for the tree $T_{\min}$.
		Apply \emph{8.} to deduce that the value of each $1$-flow on a balanced tree on $n$ vertices is not less than $\frac{4-n}{2}$.
		For $T_{\min}$ the $1$-flow is $\{\frac{4-n}{2},1\}$-flow.  Other balanced tree on $n$ vertices have at most $n-4$ leaves.
		Use \emph{8.} to deduce that the value of each $1$-flow on a balanced tree on $n$ vertices is not more than $\frac{n-4}{2}$.
		There are four nonisomorphic balanced trees with $n-4$ leaves.  $T_{\max}, T_{\opt}$ and  $S_1,S_2$.  $S_1$ is obtained from $T_{\min}$ by 
		deleting one leaf in $V_1(T_{\min})$ and adjoining one vertex of a leaf in $V_2(T_{\min})$. 
		$S_2$ is obtained from $T_{\max}$ by removing one leaf from $V_1(T_{\max})$ and from $V_2(T_{\max})$ and adjoining these two leaves to
		$v_2$ and $v_3$, respectively.
		 For $T_{\max}$ the $1$-flow is 
		$\{\frac{6-n}{2},1, \frac{n-4}{2}\}-flow$.  For $T_{\opt}$ the $1$-flow is $\{\frac{6-n}{2},\frac{8-n}{2}, 1, \frac{n-8}{2},\frac{n-6}{2}\}$-flow.
		For $S_1$ the $1$-flow is $\frac{6-n}{2},0,1\}$-flow.  For $S_2$ the $1$-flow is $\{\frac{8-n}{2},1, \frac{n-8}{2}\}$-flow.
	
		If $T$ has at most $n-5$ leaves then \emph{8} implies that the range of $1$-flow is in $[\frac{8-n}{2}, \frac{n-6}{2}]$.
	\end{enumerate}
\end{proof}

%------------------------------------------------------------------------------------------------------------------
\subsection{The range of 1-sum flows on Unicyclic graphs}
We can also find a bound for the range of a 1-sum flow on a connected unicyclic graph, i.e. a graph which is obtained from a tree by adding a single edge.  
As for trees we call a vertex of degree one a leaf, and just as for trees the number of leaves turns out to control the range of the 1-sum flows. The bound is also 
strongly dependent on whether the graph is bipartite or not, with bipartite graphs giving us a narrower range, and in each case we find graphs for which the stated bound is optimal.

\begin{theorem}\label{1flowconntr}
	Let $G=(V,E)$ be a connected unicyclic graph, with $|V|=n=|E|$, which has a 1-sum flow.  Assume that $G$ has $p\ge 0$ leaves.  
	
	Then one of the following conditions holds:
	\begin{enumerate}
		\item $p=0$. In this case $G$ is a cycle and has a 1-sum $\{\frac{1}{2}\}$-flow

		\item $p=1$. If $G$ has a $1$-sum flow then it has a 1-$[0,1]$-flow
				
		\item $p\geq 2$ and $G$ is not bipartite.  Then $G$ has a 1-sum $[1-p,p]$-flow.  
		
		This bound is optimal for the graph obtained by taking the disjoint union of  a triangle and $K_{1,p+1}$ and joining one vertex on the 
		triangle to one of the leaves of the $K_{1,p+1}$.

		\item $p\geq 2$ and $G$ is a balanced bipartite graph.    Then $G$ admits a 1-sum $[1-\lfloor\frac{p}{2}\rfloor,\lfloor\frac{p}{2}\rfloor]$-flow.  
		
		This bound is optimal for the graph obtained by taking two copies of $K_{1,p/2}$ and joining the two high degree vertices by a six vertex 
		path, giving a total of $p+8$ vertices in the graph, and then adding an edge so that the middle 4 vertices of the path form a 4-cycle.

	\end{enumerate}
\end{theorem}
\begin{proof}  
	\begin{enumerate}
		\item Set each the value on each edge to $\frac{1}{2}$.
	
		\item If $p=1$ then $G$ consists of a cycle $C$ joined to a path $P$ by a single edge $e=\{u,v\}$, where $u\in C$. The flow on the path 
			is uniquely determined, and is locally a  flow with only values 0 and 1.   If the flow on $e$ is 0 we can set the weight on every edge 
			in $C$ to $\frac{1}{2}$ and we are done.  If the flow on $e$ is 1 then $C\setminus u$ is a path with an even number of vertices, since 
			a 1-sum flow exists, and we can set the weight on a perfect matching in that path to 1 and 0 on the remaining edges, and so we have a 
			flow on $G$ with only weights 0 and 1.

		\item We inductively assume that the theorem is true for smaller $n$ and $p$.  Let $u$ and $v$ be two leaves of $G$.  If $u$ is adjacent to  
			a vertex $w$ of degree 2 then 	$G'=G\setminus \{u,w\}$ has a 1-sum $[1-p,p]$-flow, by induction on $n$, and by setting the weight on the 
			edge $\{u,w\}$ to 1 we can extend this to a 1-sum $[1-p,p]$-flow on $G$, and we can follow the same procedure if $v$ is adjacent to a vertex 
			of degree 2. Hence we can assume that $u$ and $v$ are not adjacent to vertices of degree 2.
	
			Since $G$ is not bipartite 	there exists a walk $W$ of odd length in $G$ from $u$ to $v$. 	By induction on $p$ the graph 	
			$G'=G\setminus \{u,v\}$ has a 1-sum flow of the desired range. We can now build a 1-sum flow on $G$ by setting the flow on the edges incident to
			$u$ and $v$ to 1, and then alternatingly subtract and add 1 to the weight of the edges along $w$. In this way we get a 1-sum flow on $G$, and 
			since $G'$ had two less leaves than $G$ and our modification changed each weight by at most 2, which happens if the edge was traversed
			twice by the walk $W$, we get a flow of the desired range.

		\item In this case $G$ is a balanced bipartite graph containing a single even cycle $C$.  Let $u$ and $v$ be two leaves of $G$ belonging to 
			different parts of the bipartition. If either one of them, say $u$, is adjacent to a vertex $w$ of degree 2 then $G'=G\setminus \{u,w\}$ is also 
			a balanced bipartite graph and, by induction on $n$, it has a 1-sum flow with the desired range. By setting the weight on the edge $\{w,u\}$ to 1
			and the weight on the other edge incident to $w$ to 0, extend this to 1-sum flow of the desired range on $G$. Hence we can assume that $u$ 
			and $v$ are not adjacent to vertices of degree 2.
			
			Since $u$ and $v$ are in different parts there exists a path from $u$ to $v$ in $G$ of odd length. By induction on $p$  the graph
			$G'=G\setminus \{u,v\}$ has a 1-sum flow of the desired range. We can now build a 1-sum flow on $G$ by setting the flow on the edges incident to
			$u$ and $v$ to 1, and then alternatingly subtract and add 1 to the weight of the edges along $w$. In this way we get a 1-sum flow on $G$, and 
			since $G'$ had two less leaves than $G$ and our modification changed each weight by at most 1 we get a flow of the desired range.

	\end{enumerate}

\end{proof}
%

%------------------------------------------------------------------------------------------------------------------
\subsection{The range of 1-sum flows for general connected graphs}
Our last two results give optimal bounds for the range of a 1-sum flow on very sparse graphs.  For denser graphs there is a wider variety of behaviors and we 
do not have, or expect, an optimal bound in terms of any simple graph parameter.  However, when one considers denser graphs $G$ the expectation in general 
is that the optimal range should become narrower, especially since a 1-sum $[a,b]$-flow on a spanning subgraph can be extended to a 1-sum flow on $G$ with range 
$\{0\}\cup [a,b]$ in the obvious way.  For a graph with a $k$-factor this immediately leads to a 1-sum flow with a narrow range.
\begin{lemma}
	If $G$ has a $k$-regular spanning subgraph then $G$ has a $1$-sum $[0,\frac{1}{k}]$-flow.
\end{lemma}

We know that very dense graphs have 1-factors, and that random graphs with positive density have $k$-factors for quite large values of $k$, but there are of 
course quite dense graphs which do not even have a 1-factor.  However all connected graphs have a spanning tree and using this fact our earlier results for 
trees implies a bound on the range for general graphs as well.
\begin{corollary}\label{1flowgengraph}  
	Let $G=(V,E), |V|=n, |E|=m$ be a connected graph.  Then there exists a $1$-flow  if and only $G$ is not a bipartite nonbalanced graph.
	If  a $1$-flow exists, then there exists a flow of the following type:
	\begin{enumerate}
		\item $G$ is a balanced bipartite graph with $n\ge 8$.  Then there exists an integer valued flow with values $\{2-\frac{n}{2},\ldots,\frac{n}{2}-2\}$.
		\item $G$ is nonbipartite graph.  Then there exists a flow $x$ such that $2x\in \Z^m$.  Moreover for $n\ge 6$ its values are in the interval $[5-n,n-5]$.
	\end{enumerate}
\end{corollary}
This results can be sharpened a bit by including information about the independence number $\alpha(G)$ of $G$.   In \cite{bbgks} it was proven that unless 
$G$ is a cycle, a complete graph, or a balanced complete bipartite graph it has a spanning tree the end vertices of which form an independent set in $G$. Using this we get the following.
\begin{corollary}\label{1flowbdinds}
	Let $G=(V,E)$ be a connected graph with independence number $\alpha(G)$.
	\begin{enumerate}
		\item If $G$ is $k$-regular then $G$ has a $1$-sum $\{\frac{1}{k}\}$-flow.
		\item If $G$ is not regular and not bipartite then $G$ has a $1$-sum $[1-\alpha(G),\alpha(G)]$-flow.
		\item If $G$ is not regular, but is bipartite and balanced, then $G$ has a $1$-$[\lfloor\frac{\alpha(G)}{2}\rfloor, -\lfloor\frac{\alpha(G)}{2}\rfloor]$-flow.
	\end{enumerate}
\end{corollary}

These bounds are quite far from the actual range for most graphs, since we know that for any fixed $r$ a random graph with minimum degree at least $r$ almost 
surely has an $r$-factor, and hence a 1-sum $[0,\frac{1}{r}]$-flow.

If a graph $G$ has $k$ disjoint spanning trees and we have a 1-sum flow $\omega_T$ on each tree $T$ then the average of these flows, seen as flows on $G$, will 
also be a 1-sum flow on $G$. So in this case we can reduce the bounds in Corollary \ref{1flowgengraph}  by a factor of $\frac{1}{k}$. Here we recall that Nash-Williams \cite{NW} and Tutte \cite{tu2} have characterized the graphs which have $k$ disjoint spanning trees, and so their characterization together with  Corollary \ref{1flowgengraph} give us a collection of graph classes with smaller ranges for their 1-sum flows.

In principle one could use the averaging procedure from the last paragraph for a collection of non-disjoint trees too. On one hand we might now be averaging 
several positive weights on a single edge, in which case we are no longer guaranteed to gain a factor of $\frac{1}{k}$ over the single tree bound but on the 
other hand we might have both positive and negative weights on the same edge, and the cancellation could lead to even greater gains.  It would be interesting 
to see what can be said about a flow obtained in this way by taking a random collection of spanning trees in $G$.

%---------------------------------------------------------------------
\subsection{Nonnegative $1$-flows}
Let $\Omega_n\subset \R_+^{n\times n}$ be the set of doubly stochastic matrices.  That is $A=[a_{ij}]_{i,j=1}^n$ is a nonnegative matrix such that
each row and column has sum $1$.  Denote by $\cP_n\subset \Omega_n$ the group of $n\times n$ permutation matrices.
Recall the classical result of G. Birkhoff \cite{Bir46}, which is also called Birkhoff-von Neumann theorem \cite{vNe53}.
Namely, the extreme points of doubly stochastic matrices are the permutation matrices.  

Let $\Omega_{n,s}\subset \Omega_n$ be the subset of symmetric doubly stochastic matrices.
The following result due to M. Katz \cite{Kat70}:
\begin{theorem}\label{katzthm}  
	Let $\Omega_{n,s}$ be the set of symmetric doubly stochastic matrices.  Then $A\in\Omega_{n,s}$ is an extreme point of $\Omega_{n,s}$ 
	if and only if $A=\frac{1}{2}(Q+Q\trans)$ for some permutation matrix $Q\in\cP_n$.  Equivalently, there exists   a permutation matrix  
	$P\in \cP_n$ such that $A=PBP\trans$, where $B=\diag(B_1,\ldots,B_t)$ and each $B_j$ is a doubly stochastic symmetric matrix of the following form:
	\begin{enumerate}
		\item The $1\times 1$ matrix $[1]$.
		\item $A(K_2)$.
		\item $\frac{1}{2}A(C)$, where $C$ is a cycle.
	\end{enumerate}
\end{theorem}

\begin{corollary}\label{katzthmcor} 
	Let $\Omega_{n,s,0}$ be the set of symmetric doubly stochastic matrices with zero diagonal.  Then $A\in\Omega_{n,s,0}$ is an 
	extreme point of $\Omega_{n,s,0}$ if and only if $A=\frac{1}{2}(Q+Q\trans)$ for some permutation matrix $Q\in\cP_n$ which does 
	not fix any $i\in[n]$.   Equivalently, there exists   a permutation matrix $P\in \cP_n$ such that $A=PBP\trans$, where 
	$B=\diag(B_1,\ldots,B_t)$ and each $B_j$ is a doubly stochastic symmetric matrix of the forms 2 or 3 given in Theorem \ref{katzthm}.
\end{corollary}

Let  $H$ be a simple graph.  $H$ is called $1$-factor, or perfect matching, if each connected component is $K_2$.
$H$ is called $\{1,2\}$-factor if each connected component of $H$ is either $K_2$ or a cycle.
$G$ has $1$-factor, or perfect matching, if $G$ has a spanning subgraph which is $1$-factor.
$G$ has $\{1,2\}$-factor if $G$ has a spanning subgraph which is $\{1,2\}$-factor.
\begin{theorem}\label{exist1nonfl}
	Let $G=(V,E)$ be a simple graph. Then $G$ has $1$-$[0,1]$-flow if and only if one of the following conditions hold:
	\begin{enumerate}
		\item Assume that $G$ is not bipartite.  Then $G$ has $\{1,2\}$-factor.
		\item Assume that $G$ is bipartite.  Then $G$ has $1$-factor.
	\end{enumerate}
	Furthermore, $G$ has a $1$-$(0,1]$-flow if and only if for each $e\in E$ one of the following conditions holds:
	\begin{enumerate}
		\item Assume that $G$ is not bipartite.  Then there exists $\{1,2\}$-factor of $G$ that contains $e$.
		\item Assume that $G$ is bipartite.  Then there exists $1$-factor of $G$ that contains $e$.
	\end{enumerate}
\end{theorem}
\begin{proof} 
	Suppose first that $G$ is not bipartite.    Assume that $n=|V|$. 
	View $V=[n]=\{1,\ldots,n\}$ and $E$ as a subset of all pairs $\{i,j\}$, where $i\ne j\in[n]$.
	Clearly, $G$ has $1$-$[0,1]$ if and only if there exists  $C=[c_{ij}]_{i,j=1}^n \in \Omega_{n,s,0}$,
	such that $c_{ij}=0$ if $(i,j)\not\in E$.  Corollary \ref{katzthmcor} yields that $C$ is a convex combination of $A=[a_{ij}]$ such that $a_{ij}=0$ if 
	$\{i,j\}\not\in E$.  Take such an extreme point.  Corollary \ref{katzthmcor} implies that $A$ corresponds to a $\{1,2\}$-factor of $G$.
	
	Vice versa, assume that $H$ is $\{1,2\}$-factor of $G$.  Let $\omega_H:H\to \{\frac{1}{2},1\}$ be the following flow on $H$.
	On each edge of the connected component $K_2$ of $H$ the value of $\omega_H$ is $1$.  On each edge of the cycle in $H$ the value of the 
	edge is $\frac{1}{2}$.   Extend this flow to $\hat\omega_{H}:E\to \{0,\frac{1}{2},1\}$ by letting $\hat\omega_H(e)=0$ for $e\not\in E\setminus E(H)$.
	Note that $H$ induces a unique extremal point $A(H)\in\Omega_{n,s,0}$.
	
	Assume that $G$ has $1$-$[0,1]$-flow.  Denote by $\Omega_{n,s,0}(G)$ all the symmetric doubly stochastic matrices corresponding to the $1$-$[0,1]$-flow
	on $G$.   Let $A_1,\ldots,A_M$ be all the extremal points of of $\Omega_{n,s,0}(G)$.  So $A_i=A(H_i)$ where $H_i$ is $\{1,2\}$-factor of $G$.
	So any $1$-$[0,1]$-flow is a convex combination of $A(H_1),\ldots,A(H_M)$.  Suppose there exists $1$-$(0,1]$-flow $\omega$ on $G$.  Let $e\in E$.
	Since $\omega(e)>0$ it follows that $e$ is contained in some $H_i$.  Vice versa, suppose $H_1,\ldots,H_M$ are all $M$ $\{1,2\}$-factors of $G$.
	Assume that each $e\in E$ is contained in some $H_i$.  Consider the $1$-flow $\omega=\frac{1}{M}\sum_{i=1}^M \hat\omega_{H_i}$.
	Then $\omega$ is $1$-$(0,1]$-flow.   
	
	Assume now that $G$ is a bipartite graph.  So $1$-$\R$-flow exists if and only if $G$ is balanced bipartite graph $G=(V_1\cup V_2,E)$.
	Let $V_1=\{u_1,\ldots,u_n\}, V_2=\{v_1,\ldots,v_n\}$.   So each edge $e\in E$ is of the form $\{u_i,v_j\}$.  Then $1$-$[0,1]$-flow $\omega$ on $G$ 
	corresponds to $A=[a_{ij}]_{i,j=1}^n\in\Omega_n$ where $a_{ij}=0$ if $\{u_i,v_j\}\not\in E$.  Recall Birkhoff's theorem which shows that $\cP_n$
	is the set of extreme points on $\Omega_n$.  
	
	Assume first that $G$ has $1$-$[0,1]$-flow $\omega$.  Then $A\in \Omega_n$ represents $\omega$.  So $A=\sum_{j=1}^M a_j P_j$ where each 
	$P_j\in\cP_n$, $a_j>0$ and $\sum_{j=1}^M a_j=1$.  Hence each $P_j$ represents $1$-factor of $G$.  Vice versa, assume that $G$ has $1$-factor $H$.  
	The arguments above imply that $\hat\omega_H$ is $1$-$[0,1]$-flow on $G$.  As in the non-bipartite graph we deduce that $G$ has $1$-$(0,1]$-flow 
	if and only if each edge is covered by some $1$-factor of $G$.
\end{proof}

The fundamental works of Tutte give necessary and sufficient conditions for existence $1$ and $\{1,2\}$ factors \cite{Tut47,Tut53}.

\begin{corollary}\label{exist01fdelge2} 
	Let $G$ be a graph and $\delta(G)\geq 2$. If $G$ has no even cycle,  then $G$ admits a $1$-$[0, 1]$-flow.
\end{corollary}
\begin{proof} 
	We claim that $G=(V,E)$ has a $\{1,2\}$-factor. We prove this claim by induction on $n=|V(G)|$.   For $n=3$ the claim is trivial. 
	Consider the block decomposition of $G$.  It is well-known that every block of $G$ is $K_2$ or an odd cycle, see \cite{Wes00}. 
	Now, choose a leaf block of $G$. Obviously, it is an odd cycle $C$ on 2l+1 vertices.  Suppose first that $C$  has
	 a common vertex $v$, with another odd cycle $C'$.  Remove all vertices of $C$ except $v$.  The remaining graph $G'$ satisfies the assumption 
	 of the corollary. By the induction hypothesis $G'$ has a $\{1,2\}$-factor. The subgraph of $C$ on $2l$ vertices has  a $1$-factor.  Hence $G$ 
	 has a $\{1,2\}$-factor.
	
	It is left to discuss the case where the leaf cycle has one vertex $v$ of degree $3$ which is common with 
	a $K_2$-block.   Consider the the shortest path, $P$, between $v$ and another vertex of degree at least 3, say $w\not \in V(C)$. 
	Remove all the vertices on $C$ and the path $P$ except the vertex $w$.  The remaining graph $G'=(V',E')$ has a $\{1,2\}$-factor by induction.
	Consider now the subgraph $G_1$ of $G$ on the vertices $V\setminus V'$.  If the length of path is odd then $G_1$ has a $\{1,2\}$-factor
	consisting of $C$ and a matching, where the matching may be empty.  If $P$ is even then $G_1$ has a $1$-factor.  Hence $G$ has a $\{1,2\}$-factor.  
\end{proof}

%------------------------------------------------------------------------------------------------------------------
\subsection{Existence of 1-sum $[-1,1]$-flows on graphs with $\delta(G)\ge 2$} 
In this section we assume that $G$ is connected graph with the minimal degree $\delta(G)$ at least $2$.
We believe that it would be interesting to characterize the graphs which admit a 1-sum $[-1,1]$-flow. This class clearly extends the class of graphs which 
have a  1-sum $[0,1]$-flow, but the inclusion of negative edge weights adds more flexibility.  Lemma \ref{Farkaslem} 
gives a necessary and sufficient conditions on the existence of these flows but at the moment we do not have a good interpretation of this result
in terms of structural properties of the graphs.

However, we can show that not all graphs have a 1-sum $[-1,1]$-flow, and in fact that given an integer $t$  there are graphs of arbitrarily high 
edge-connectivity which does not have a 1-sum $[-t, \infty)$-flow.

We start out with a simple example and then proceed with the generalization to higher connectivity.  
\begin{figure}[h!t]
	\begin{center}
		\includegraphics[width=\textwidth]{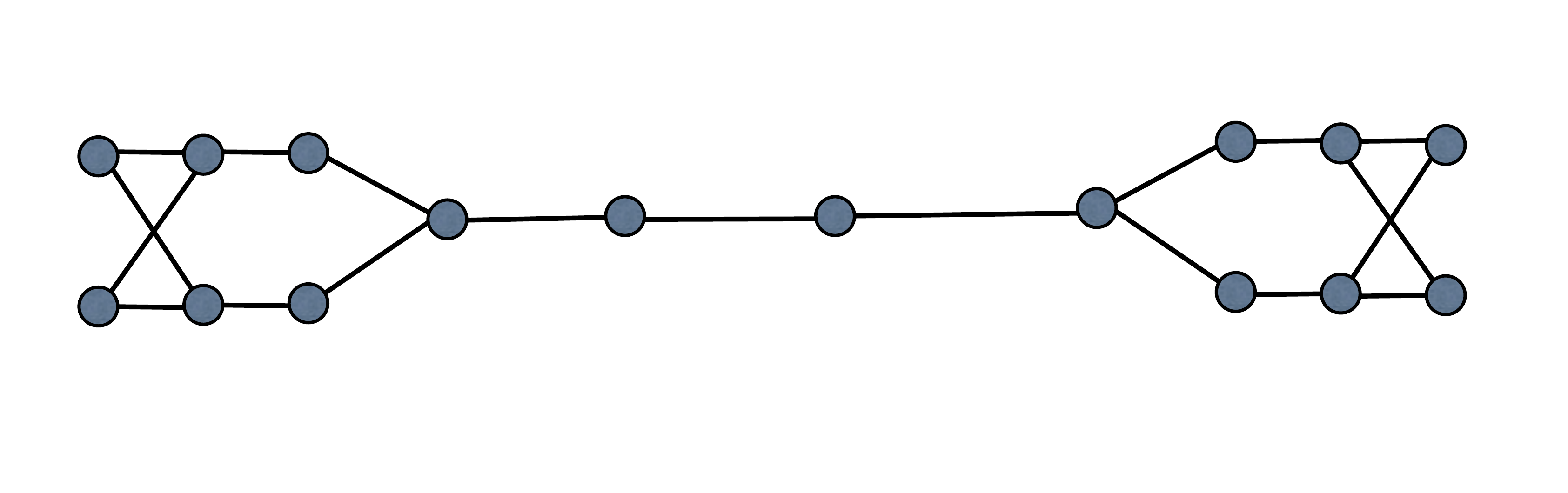}
		\caption{A bipartite graph with no 1-sum $[-1,1]$-flow.}
		\label{fig1}
	\end{center}
\end{figure} 
\begin{example}\label{examp1}  
	The graph $G$ on $16$ vertices with $\delta(G)=2,\Delta(G)=3$ given in Figure \ref{fig1}, 
	does not have $1$-sum $[-1,1]$ flow. A direct computation shows that the center edge of $G$ has weight 2 in all 1-sum flows and that  
	the narrowest range is given by a $1$-sum $[-1,2]$-flow.
\end{example}

\begin{example}\label{examp2} 
	For two positive integers $t$ and $s$,  there is an $s$-edge connected bipartite graph $G$
	 which admits a $1$-sum $\mathbb{R}$-flow but  admits no 1-sum $[-t, \infty)$-flow.

	Consider two disjoint copies of $K_{s,s(1+t)+1}$. Call the vertex parts of the first one by $(X,Y)$ and the second one by $(X',Y')$,
	where $|X|=|X'|=s$ and $|Y|=|Y'|=s(1+t)+1$. Choose an arbitrary vertex $v\in X$ and join $v$ to all vertices in $X'$ with the edges 	
	$e_1, \ldots, e_s$ and call  the resulting graph by $G$. We claim that $G$ is the desired graph. Clearly, $G$ is an $s$-edge connected 
	bipartite graph. Note that $G$ is a balanced bipartite graph and so it admits a $1$-sum $\mathbb{R}$-flow. By contradiction assume 
	that $f$ is a $1$-sum $[-t, \infty)$-flow of $G$. Then we have 
	$$s=\sum_{i=1}^{s} s(x_i)=\sum_{1\leq i\leq s, 1\leq j\leq s(1+t)+1}f(x_iy_j)+ \sum_{i=1}^{s}f(e_i)=s(1+t)+1+\sum_{i=1}^{s}f(e_i),$$ 
	where $s(x_i)$ is the sum values of all edges incident with $x_i$. This implies that $\sum_{i=1}^{s}f(e_i)=-st-1$. We know that for each 
	$i$, $f(e_i)\geq -t$, a contradiction.
\end{example}

\begin{problem} 
	Characterize the graphs which admit a 1-sum $[-1,1]$-flow.
\end{problem}

%---------------------------------------------------------------------------------------------------------------------------------------------------------------------------------------------------------------------------
\section{$1$-sum $L$-flows when $L$ is not an interval}
As mentioned in the introduction the problem of finding a $1$-sum $L$-flow when $L$ is an interval is a linear programming problem. As soon as $L$ is not an 
interval we are no longer working with a convex problem and many of the tools we have  used so far do not apply.  Nonetheless we shall prove some results 
for two cases of this type. First we will consider the real line with the single point 0 removed, a second we will look at the case when $L$ consist of just a finite 
list of  real numbers.

%----------------------------------------------------------------------------------------
\subsection{$1$-$\R^*$-flows}

In \cite{akb} the following question was proposed.
\begin{question}
Determine a necessary and sufficient condition under which a bipartite graph admits a
$1$-sum $\mathbb{R}^*$-flow or a $1$-sum $\mathbb{Z}^*$-flow.
\end{question}
In this section we give an answer to this question.

It is not hard to see that if a graph $G$ admits a $1$-sum $\mathbb{Z}$-flow, then the order of $G$ should be even.
In [4] it has been proved that a connected  bipartite graph admits a $1$-sum $\mathbb{R}$-flow if and only if it is balanced.

\begin{theorem} 
	Let $G$ be a connected balanced bipartite graph. Then $G$ admits a $1$-sum $\mathbb{R}^*$-flow if and only if the removing of every cut edge 
	of $G$ does not make a balanced bipartite connected component.
\end{theorem}
\begin{proof} 
	First assume that  $G$ admits a $1$-sum $\mathbb{R}^*$-flow, say $\omega$, and $e$ is a cut edge its removing makes a balanced bipartite 
	connected component. Call this component by $H$. Assume that $(X,Y)$ be two vertex parts of $H$ and $|X|=|Y|$. We have 
	$$|X|=\sum_{v\in X}s(v)=\sum_{v\in Y}s(v)+\omega(e)=|Y|+\omega(e),$$
	where $s(v)$ denotes the sum of the values of all incident edges to $v$. This implies that $\omega(e)=0$, a contradiction.

	Now, assume that the  removing of every cut edge does not make a balanced bipartite connected component. Let $E(G)=\{e_1, \ldots ,e_m\}$ and 
	for $i=1, \ldots ,m$, $W_i\subset \mathbb{R}^m $ is the set of all $0$-sum flows of $G$ in which the value of $e_i$ are zero.  Let $V\subset \mathbb{R}^m $ 
	be the set of all $0$-sum flows of $G$. Clearly, $V$ and $W_i$ are vector spaces over $\mathbb{R}$. By Theorem 3 of [4], there is a $1$-sum  
	$\mathbb{R}$-flow for $G$.

	If $V \not  \subset \bigcup_{i=1}^m W_i$, then there exists a $0$-sum  $\mathbb{R}^*$-flow $\omega'$ of $G$. It is obvious that there exists a suitable 
	real number $a$ such that $\omega+a\omega'$ is a  $1$-sum  $\mathbb{R}^*$-flow and we are done.

	Now, assume that there exists $J\subseteq \{1, \ldots, m\}$ such that for every $j\in J$, $V\neq W_j$ and for any $j\in \{1,\ldots, m\}\setminus J$,
	$V= W_j$. Since $\mathbb{R}$ is infinite, it is well-known that $V\not \subset \bigcup_{j\in J}W_j$. So, there exists a vector $\alpha \in V$,
	such that the $j$th component of $\alpha$ is non-zero for every $j\in J$. Now, let $j\in \{1,\ldots, m\}\setminus J$. If $e_j$ is not a cut edge, then it is contained 
	in an even cycle. If we assign $1$ and $-1$ to all edges of this cycle, alternatively and assign $0$ to all other edges of $G$, we obtain a vector in 
	$V\setminus W_j$, a contradiction. Hence $e_j$ is a cut edge of $G$. By assumption $G\setminus \{e_j\}$ has a non-balanced bipartite component $H$. Note that
	since $G$ is balanced and $H$ is not balanced, the other component, $H'$, is not balanced too. Let $H=(A,B)$ be two vertex parts of $H$ and $|A|<|B|$. Without 
	loss of generality assume that $v\in A$ and $e_j$ is incident with $v$. Assign 1 to every vertex in $(A\setminus \{v\})\cup B$ and assign $|B|-|A|+1$ to $v$. Then 
	by Theorem 3 of \cite{akb}, $H$ admits a flow such that $s(v)=|B|-|A|+1$ and $s(x)=1$, for each $x\in V(H)\setminus \{v\}$. Similarly, if $H'=(A',B')$ and $|A'|<|B'|$ 
	and $e_j$ is incident with $v'\in A'$, then there exists a flow for $H'$ such that  $s(v')=|B'|-|A'|+1$ and $s(x)=1$, for each $x\in V(H')\setminus \{v'\}$. Since $G$ is 
	balanced, we have $|A|+|B'|=|A'|+|B|$. This yields that  $s(v)=s(v')$. Now, assign $|A|-|B|$ to $e_j$ to obtain a $1$-sum flow for $e_j$. Note that since $V=W_j$, in 
	any $1$-sum  $\mathbb{R}$-flow  of $G$, the value of $e_j$ should be $|A|-|B|$, which is non-zero. It is not hard to see that there exists a suitable $a$ such that 
	$\omega+a\alpha$ is a $1$-sum  $\mathbb{R}^*$-flow of $G$, as desired.
\end{proof}
Let $G$ be a $2$-edge connected bipartite graph and $a<0$ and $b>0$ be two real numbers and $L=(a,b)$. Then $G$ admits a $1$-sum $L$-flow if and only if 
$G$ admits a $1$-sum $L^*$-flow. To see this, by a theorem, $G$ has a $0$-sum $\mathbb{R}^*$-flow, say $\omega'$. Let $\omega$ be a  $1$-sum $L$-flow of 
$G$. Then if $\epsilon$ is small enough,  $\omega+\epsilon \omega'$ is a  $1$-sum $\mathbb{R}^*$-flow of $G$.

There are many possible variations of these questions
\begin{question} 
	What is the difference between the graphs which admit a 1-sum $[-1,1]$-flow, a 1-sum $(-1,1)$-flow and a  1-sum $[-1,1]^*$-flow?
\end{question}

%------------------------------------------------------------------------------------
\subsection{$1$-sum flows with a finite list}
We can also let $L$ be a finite list of allowed values, bringing us closer to the situation in the classical study of nowhere-zero flows of graphs.

Consider $\gam$-flow on a connected $G$.  Assume that $\gam\in\C^V$ and consider $\gam$-$\C$-flow on $G$.  Lemma \ref{existgamRflow} applies also 
to $\C$-flows. Let $L=\{t_1,\ldots,t_k\}$ be a subset of $\C$ of cardinality $k$.  We now interested in the problem when there exists $\gam$-$L$-flow.
This is a problem in algebraic geometry.   Define
\[P_L(z):=\prod_{j=1}^k (z-t_j).\]
So, what we are asking for is a solution of  the linear system \eqref{sysomeggam} and the polynomial conditions
\[P_L(\omega(e))=0 \textrm{ for each } e\in E.\]
Since $r=\rank A(G)\in\{n-1,n\}$ we have an overdetermined system of equations with $m-r$ parameters which will have to satisfy 
$m$ polynomial conditions.  For relatively small graphs the existence of a solution can be determined by standard computer algebraic 
software like Mathematica.

In what follows we discuss very special flows and $L$ using graph theory results.
\begin{theorem} \label{regular } 
	Let $k$ be a positive integer and $G$ be a connected $k$-regular graph of order $n$. Then the following hold:
	\begin{enumerate}
		\item If $k$ is odd, then $G$ admits a $1$-sum $\{-1, 0, 1\}$-flow.
		\item If $k\equiv 2\, (mod\, 4)$ and $n$ is even, then $G$ admits a $1$-sum  $\{-1, 0, 1\}$-flow.
	\end{enumerate}
\end{theorem}

\begin{proof} 
	\begin{enumerate}
		\item First we assign a  bipartite graph $H$ to $G$. Suppose that
			$V(G)=\{1,\dots,n\}$ and let $H$ be a bipartite graph with two
			parts $\{x_1,\dots,x_n\}$ and $\{y_1,\dots,y_n\}$. Join $x_i$ and
			$y_j$ if and only if two vertices $i$ and $j$ are adjacent in $G$.
	
			Since $H$ is a $k$-regular bipartite graph, the edges of $H$ can be decomposed into $k$, $1$-factors, $F_1, \ldots , F_k$.
			Assign $\frac{(-1)^{i-1}}{2}$ to all edges of $F_i$, $1\leq i\leq n$.
			So, for each vertex $v\in V(H)$, we have $s(v)=\frac{1}{2}$.
			For two adjacent
			vertices $v_i$ and $v_j$ in $G$,  assume $e_{ij}$ is the edge between $v_i$ and $v_j$. Let $a_{ij}$ be the value of the edge $x_iy_j$, $1\leq i,j\leq n$.
			Assign the value $b_{ij}=a_{ij}+a_{ji}$ to $e_{ij}$. By our assumption,
			$b_{ij}\in\{-1, 0,1\}$. We have
			$$\sum_{y_j\in N(x_i)}a_{ij}=0\ ,\ \sum_{x_j\in N(y_i)}a_{ji}=0.$$
			It is not hard to see that using this assignment we find a $1$-sum $\{-1,0,1\}$-flow for $G$.

		 \item Since $k$ is even, $G$ is an Eulerian graph and so it is $2$-edge connected. Now, by Theorem 3.10, Part (ii) of \cite{kann},
			 the edges of $G$ can be decomposed into two spanning $(2k+1)$-regular graphs $G_1$ and $G_2$. By Part (i), $G_1$ has
			 a  $1$-sum $\{-1,0,1\}$-flow. Now, assign $0$ to all edges of $G_2$. So, $G$ admits a $1$-sum $\{-1,0,1\}$-flow, as desired.
	\end{enumerate}
\end{proof}

\begin{question} 
	Let $k$ be a positive integer divisible by $4$. Is it true that every connected $k$-regular graph of even order admits a 1-sum $\{-1,0,1\}$-flow? 
\end{question} 

\begin{problem} 
	Characterize the graphs which admit a 1-sum $\{-1,0,1\}$-flow.
\end{problem}

Next we recall the following interesting result.
\begin{theorem}\label{kano}{ \em \cite{kano}}
	\textit{Let $r \geq 3$ be an odd integer and let $k$ be an integer such that $1 \leq k \leq \frac{2r}{3}$. Then every $r$-regular graph
	has a $[k-1, k]$-factor each component of which is regular.}
\end{theorem}

\begin{theorem}
	Let $r\geq 5$ be an odd positive integer. Then every $r$-regular graph admits a $1$-sum $3$-flow. Moreover, every $2$-edge connected 
	$r$-regular graph admits a $1$-sum $2$-flow.
\end{theorem}

\begin{proof} 
	First let $r=5$. By Theorem \ref{kano}, $G$ has a $[2,3]$-factor $H$ whose each component is regular.
	Now, assign $-1$ to any edge in $E(G)\setminus E(H)$ and $1$ to all edges of any $3$-regular component and $2$  to all edges of any
	$2$-regular component to obtain a $1$-sum $3$-flow.

	Now, let $r=2t+1\geq 7$.
	We have $1\leq t+1\leq \frac{2(2t+1)}{3}$. By Theorem \ref{kano}, $G$ has a $[t, t+1]$-factor whose each component is regular.
	Let $H$ be the union of all $t$-regular components  and $K$ be the union of  all $(t+1)$-regular components of $G$. First assume that $t$ is odd.
	Assign $1$ to all edges in $E(G)\setminus (E(H)\cup E(K))$. Also,
	Assign $-1$ to all edges of $H$. Since $t+1$ is even, by Petersen Theorem, $K$ has a $4$-regular factor say $L$.  Since $L$ is a union
	of two $2$-factors, it admits a $-2$-sum $3$-flow. Now, assign $-1$ to all edges of $E(K)\setminus E(L)$ to obtain a $1$-sum $3$-flow
	for $G$. Now, let $t$ be even. Assign $-1$ to all  edges in $E(G)\setminus (E(H)\cup E(K))$. Assign $1$ to all edges of $K$. Since $t$ is even,
	$H$ has a $2$-factor, say $L$. Assign $1$ to all edges in $E(H)\setminus E(L)$ and $2$ to all edges in $E(L)$ to obtain a $1$-sum
	$3$-flow for $G$.
	
	The last part is an immediate consequence of Theorem 8 of \cite{zar}.
\end{proof}

\begin{theorem} 
	Let $r\geq 3(r\neq 5)$ be an odd positive integer. Then every $2$-edge connected $r$-regular graph admits a  $0$-sum $3$-flow.
\end{theorem}

\begin{proof} 
	Let $G$ be an $r$-regular graph. We consider three cases:
	\begin{enumerate}
		\item $r=3t+0$. By Theorem 3.10, Part (v) of \cite{kano}, $G$ has a $t$-factor. Thus $E(G)$ can be decomposed into one $t$-factor and one $2t$-factor.
			Assign $2$ and $-1$ to each edge of  $t$-factor  and $2t$-factor, respectively to obtain a $0$-sum $3$-flow for $G$.

		\item $r=3t+1$. Since $r$ is odd, $t+1$ is odd. By Theorem 3.10, Part (v) of \cite{kano}, $G$ has a $(t+1)$-factor.
			By Petersen's Theorem $E(G)$ can be decomposed into one $(t+1)$-factor, one $(2t-4)$-factor and two $2$-factors $F_1$ and $F_2$. Now, assign 
			$2$, $-1$, $-2$ and 	$-1$ to each edge of $(t+1)$-factor, $(2t-4)$-factor,  $F_1$ and $F_2$, respectively, to obtain a $0$-sum $3$-flow for $G$.

		\item $r=3t+2$. Since $r$ is odd, $t+2$ is odd. By Theorem 3.10, Part (v) of \cite{kano}, $G$ has a $(t+2)$-factor. So by Petersen's Theorem $E(G)$ can 
		be decomposed into one $(t+2)$-factor, one $(2t-4)$-factor and one $4$-factor. Now, assign $2, -1$ and $-2$ to each edge of $(t+2)$-factor, $(2t-4)$-factor 
		and $4$-factor, respectively to obtain a $0$-sum $3$-flow for $G$.
	\end{enumerate}
\end{proof}

\begin{conjecture}
	Every $2$-edge connected $5$-regular graph admits a $0$-sum $3$-flow.
\end{conjecture}

It is not hard to see if $e$ is a cut edge of a graph $G$, then in any $0$-sum $k$-flow of $G$, the value of $e$ should even. Now, let $r$ be an odd positive integer 
and $G$ be an $r$-regular graph containing a vertex $v$ such that all edges incident with $v$ is a cut edge.  Thus $G$ does not admit a $0$-sum $4$-flow. In 
\cite{akb} and \cite{akb3}, it was proved that every $r$-regular graph ($r\geq 3$) admits a $0$-sum $5$-flow.

%---------------------------------------------------------------------------------------------------------------------------------------------------------------------------------------------------------------------------
\bibliographystyle{plain}

\end{document}